\documentclass[12pt]{article}
\usepackage{hyperref,latexsym,amssymb,amsmath,amsthm}
\let\BFseries\bfseries\def\bfseries{\BFseries\mathversion{bold}} 
\def\E{{\mathbb E}}
\def\N{{\mathbb N}}
\def\P{{\mathbb P}}
\def\R{{\mathbb R}}
\def\Z{{\mathbb Z}}
\def\dd{\delta}

\def\eps{\varepsilon}
\newcommand{\pr}[1]{\P\left(#1\right)}
\newcommand{\pin}{\mathfrak p}
\newcommand{\ssup}[1] {{\scriptscriptstyle{({#1})}}}
\newcommand{\cL}{\mathcal L}
\newcommand{\cF}{\mathcal F}

\newcommand{\sfrac}[2]{\mbox{$\frac{#1}{#2}$}}
\newcommand{\inftwo}[2]{\inf_{\substack{#1 \\ #2}}} 

\theoremstyle{plain}
\newtheorem{lem}{Lemma}
\newtheorem{thm}{Theorem}
\newtheorem{prop}{Proposition}

\newtheorem{rem}{Remark}
\newcommand{\ind}{1\hspace{-0.098cm}\mathrm{l}}

\begin{document}
\title{Persistence probabilities for an integrated random walk bridge}
\author{Frank Aurzada\footnote{Technische Universit\"at Berlin, Institut f\"ur Mathematik, Sekr. MA 7-5, Stra\ss e des 17.\ Juni 136, 10623 Berlin, Germany,
aurzada@math.tu-berlin.de}, Steffen Dereich\footnote{Westf\"alische Wilhelms-Universit\"at, Institut f\"ur Mathematische Statistik, Einsteinstr.~62, 48149 M\"unster, Germany, steffen.dereich@wwu.de}, and Mikhail Lifshits\footnote{St.\ Petersburg State University, 198504 Stary Peterhof, Dept.\ of Mathematics and Mechanics, Bibliotechnaya pl., 2,
Russia, lifts@mail.rcom.ru}}
\date{\today}
\maketitle

\begin{abstract}
 We prove that an integrated simple random walk, where random walk and integrated random walk are conditioned to return to zero, has asymptotic probability $n^{-1/2}$ to stay positive. This question is motivated by so-called random polymer models and proves a conjecture by Caravenna and Deuschel.
\end{abstract}



\section{Introduction and main result}
\subsection{Introduction}
This paper considers the so-called persistence probability, i.e.
$$
 \pr{ A_1\geq 0, \ldots, A_{n}\geq 0}\approx n^{-\theta},\qquad n\to\infty,
$$
where $A$ is some stochastic process and the number $\theta$ is called persistence exponent. The problem is also called one-sided exit problem or survival probability problem. Problems of this type have experienced quite some recent attention, see e.g.\ \cite{vlad1,aurzadadereich1,dembogao,aurzadafbm,aurzadabaumgarten,vlad2,baumgarten2,baumgarten3} and the recent survey paper \cite{assurvey}.

These probabilities have a couple of applications to problems in theoretical physics as well as to other questions in probability. We refer to the mentioned survey \cite{assurvey} and to a survey article on the related physics literature \cite{majumdar} for details.

 The particular problem that we treat here is motivated by a connection to so-called random polymer models, see Section~1.5 in \cite{caravennadeuschel2009}. Here, $A$ will be an integrated random walk, where the pair of random walk and integrated random walk is conditioned to return to the origin. This is supposed to model a polymer chain with Laplace interaction and zero boundary conditions.

Let us be more precise and introduce the relevant notation. Let $(X_i)_{i\in\N}$ be a sequence of independent symmetric Bernoulli random variables. We consider the simple 
random walk $S_n:=\sum_{i=1}^n X_i$ and the respective integrated random walk $A_n:=\sum_{i=1}^n S_i$ for $n\in\N_0:=\{0,1,\dots\}$. Note that the paired process $(S_n,A_n)_{n\in\N_0}$ is Markovian and that it can return to $(0,0)$ only at the times $4 n$, $n\geq 1$. Our main theorem is as follows.

\begin{thm} \label{thm:main} When $n\to\infty$ we have
\begin{equation} \label{eqn:mainht2}
   \pr{ A_1\geq 0, \ldots, A_{4n}\geq 0 \left| A_{4n}=S_{4n}=0 \right. } \approx n^{-1/2}.
\end{equation}
\end{thm}

This proves the conjecture by Caravenna and Deuschel (see \cite{caravennadeuschel2009}, (1.22) on p.\ 2396).

Let us give a couple of remarks. The unconditioned probability also has been subject to a number of studies: The first is due to Sinai \cite{sinai1991} who showed that, with the notation above,
\begin{equation}
\pr{ A_1\geq 0, \ldots, A_{4n}\geq 0 } \approx n^{-1/4}. 
\end{equation}
This result was subsequently refined by \cite{vlad1,aurzadadereich1,dembogao,vlad2} to the end that
$$
\pr{ A_1\geq 0, \ldots, A_{4n}\geq 0 } \sim c n^{-1/4},
$$
see \cite{vlad2}, Theorem~1, which extends to other types of random walks.

We remark that the result in Theorem~\ref{thm:main} is in contrast to conditioning only on $S_{4n}=0$, where the rate is again
$$
\pr{ A_1\geq 0, \ldots, A_{4n}\geq 0 ~|~S_{4n}=0} \approx n^{-1/4},
$$
see \cite{vlad2}, Proposition~1, where our problem is also mentioned. That is, conditioning on $S_{4n}=0$ does not change the rate from the unconditioned rate.

The remainder of this paper is structured as follows: The next subsection gives an overview of the method of proof. In Section~\ref{sec:llt} we derive a local limit theorem for the (unconditioned) process $(S_n,A_n)$, which we could not locate in the literature. A similar local limit theorem was proved in \cite{caravennadeuschel2009} (see Proposition~2.3 there) for random walks with $X_1$ having a continuous distribution and under an appropriate integrability assumption. In Section~\ref{sec:scaling}, we show some scaling properties of the process $(S_n,A_n)$ under the condition $A_1\geq 0, \ldots, A_{n}\geq 0$. Further, Section~\ref{sec:cltpinned} contains a CLT for the process $(S_n,A_n)$ ``pinned'' at some final value: $(S_n,A_n)=(p_n,q_n)$ with $(n^{-1/2}p_n,n^{-3/2}q_n)\to (p,q)$; we show that the suitably scaled law of this process converges to the law of a pair of Brownian motion and integrated Brownian motion, conditioned to end at $(p,q)$. This result may be of independent interest. Finally,  Section~\ref{sec:proofmainthm} contains the proof of the main result.

\subsection{Overview of the proof and notation}

Throughout we use the following notation
$$
   \Omega_n^+:=\{ A_j\ge 0, 1\le j\le n \}.
$$
We let
\begin{equation} \label{Dn}
 D_n:=\{\ell=(\ell_1,\ell_2)\in \Z^2:\ \ell_1=n\mod 2,\quad \ell_2=\tfrac{n(n+1)}{2}\mod 2 \}
\end{equation}
denote the set of all possible values of $(S_n,A_n)$.

We will use the notation of an adjoint process which is gained via time reversion. Depending on parameters $N\in \N$ and $(s_N,a_N)\in\Z^2$ we define the adjoint process  $(\bar S_n^\ssup{N},\bar A_n^\ssup{N})_{n=0,\dots,N}$ via
$$
\bar S^\ssup{N}_0= s_N \text{ and }  \bar A^\ssup{N}_0= a_N
$$
and the equations
$$
\bar S^{\ssup N}_{n+1}= \bar S^\ssup{N}_n - X_{N-n} \text{ and } \bar A^\ssup {N}_{n+1}= \bar A^\ssup{N}_{n}-\bar S^\ssup{N}_n.
$$
By construction, one has for $n,m\in\{0,\dots,N\}$ 
\begin{equation} \label{eqn:adjointprop}  \{(\bar S^\ssup{N}_{N-n},\bar A^\ssup{N}_{N-n}) =(S_n,A_n)\} = \{(\bar S^\ssup{N}_{N-m},\bar A^\ssup{N}_{N-m}) =(S_m,A_m)\}
\end{equation}
meaning that the time reversed adjoint process and the original process either agree or disagree for all times $n=0,\dots,N$. In particular, the event of accordance is equal to $\{(\bar S^\ssup{N}_{N},\bar A^\ssup{N}_{N})=(0,0)\}$ and $\{(S_N,A_N)=(s_N,a_N)\}$. The process $(\bar S^\ssup{N}_n,\bar A^\ssup{N}_n)$ is Markovian and the definition can be extended in a canonical way to the time index $\N_0$.


The strategy of the proof is as follows: We partition the time frame $4n$ into three periods: in the first $n$ steps, the process is observed under the conditioning. The same is done in the last $n$ steps. Then in the middle $2n$ time steps one has to use (upper bound) or ensure (lower bound) that the two ends meet.

For the upper bound, we observe that, on $\Omega_{4n}^+$ and with $A_{4n}=S_{4n}=0$, one has $\Omega_{n}^+$ and the same condition for the adjoint process on the final $n$ time steps. Conditioning on the first and last $n$ time steps, in the middle piece -- consisting of $2n$ time steps -- one again observes a pair of random walk and integrated random walk, however, starting and terminating at certain values. The probability of this pair starting and ending at certain values is governed by a local limit theorem.

The lower bound is more complex. Here one has to ensure that during the first (and last, respectively) $n$ steps one ends with $(S_n,A_n)$ in a target zone $ [a n^{1/2}, b n^{1/2}]\times [a n^{3/2}, b n^{3/2}]$, where $a, b$ are appropriate positive constants. This can be shown by analyzing the scaling of $S$ and $A$ under the conditioning. To ensure that both ends meet, we prove a CLT for the process $(S_n,A_n)$ which is ``pinned'' at the beginning and at the end (by values that are in the above target zone). This CLT helps us to transfer the question of positivity of the second component to the same question for the limiting process (that is, Brownian motion and its integrated counterpart), where the question of positivity is easily solved.

For this purpose, we will need the notion of Brownian motion $B=(B_t)_{t\geq 0}$ and integrated Brownian motion $I=(I_t)_{t\geq0}$ defined as
$$
I_t := \int_0^t B_s \,{\rm d} s.
$$
Note that the paired process $\Gamma=(\Gamma_t)_{t\geq 0}=(B_t,I_t)_{t\geq0}$  is a centered Gaussian Markov process and $\Gamma_1$ has the two dimensional Lebesgue density
\begin{equation} \label{eqn:defng}
   g(x,y)=
   \frac {\sqrt{3}}{\pi} \ \exp\left\{  - 2 x^2 +6 x y-6 y^2\right\}.
\end{equation}

A main ingredient of our proofs will be the following local limit theorem for the simple random walk and its integrated version, which we could not locate in the literature.

\begin{prop} \label{prop:llt}
\[
   \lim_{n\to\infty} \sup_{\ell\in D_n} \left|\frac{n^2}{4}\P((S_n,A_n)=\ell) 
   - g\left(\ell_1/\sqrt{n},\ell_2/n^{3/2} \right) \right| =0,
\]
where $g$ is defined in $(\ref{eqn:defng})$.
\end{prop}

The proof of this fact is given in Section~\ref{sec:llt}. Then, in Section~\ref{sec:scaling}, we prove the results concerning the scaling of $A$ and $S$ under the conditioning needed in the proof of the lower bound. Further, the CLT for pinned process is formulated and proved in Section~\ref{sec:cltpinned}. Finally, in Section~\ref{sec:proofmainthm} we give the proof of our main theorem.

\section{Local limit theorem for $(S_n,A_n)$} \label{sec:llt}
\subsection{Proof of the local limit theorem (Proposition~\ref{prop:llt})}
We start the proof with a general representation of probabilities through its characteristic function.
Let $Y\in\Z^2$ be an integer random vector and
\[
  f(t)=\E e^{i(t,Y)}= \sum_{k\in \Z^2} e^{i(t_1k_1+t_2k_2)}\P(Y=k), \qquad t=(t_1,t_2)\in\R^2
\]
its characteristic function. Then for any $\ell\in\Z^2$
\begin{eqnarray*}
  &&\int_{-\pi}^\pi\int_{-\pi}^\pi f(t) e^{-i(t_1\ell_1+t_2\ell_2)} dt_1dt_2 \\
  &=& \sum_{k\in \Z^2} \P(Y=k) \int_{-\pi}^\pi\int_{-\pi}^\pi 
     e^{i(t_1(k_1-\ell_1)+t_2(k_2-\ell_2))}dt_1dt_2
\\
   &=& \sum_{k\in \Z^2} \P(Y=k) \int_{-\pi}^\pi e^{i t_1(k_1-\ell_1)}dt_1  
   \  \int_{-\pi}^\pi e^{i t_2(k_2-\ell_2)}dt_2
\\
   &=& \sum_{k\in \Z^2}  \P(Y=k) (2\pi)^2\ 1_{\{k_1=\ell_1\}}\ 1_{\{k_2=\ell_2\}}
   =  (2\pi)^2  \P(Y=\ell).
\end{eqnarray*}
We conclude that
\begin{equation}\label{pf}
   \P(Y=\ell) =  (2\pi)^{-2}  \int_{-\pi}^\pi\int_{-\pi}^\pi 
                 f(t) e^{-i(t_1\ell_1+t_2\ell_2)} dt_1dt_2 .
\end{equation}
For symmetrically distributed random variables,
\begin{equation}\label{pfcos}
   \P(Y=\ell) =  (2\pi)^{-2}  \int_{-\pi}^\pi\int_{-\pi}^\pi 
                 f(t) \cos(t_1\ell_1+t_2\ell_2) dt_1dt_2 .
\end{equation}

In our case $Y_n=(S_n,A_n)$ with $S_n=\sum_1^n X_j$ and  $A_n=\sum_1^n j X_j$
where $(X_j)$ are independent Bernoulli variables.
Hence,
\[  
   f_{Y_n}(t) = \prod_{j=1}^n \cos(t_1+jt_2).
\] 

Let us discuss a periodicity property of the integrand in \eqref{pfcos}. For characteristic function we
have 
\begin{eqnarray*}
   f_{Y_n}(t_1+\pi,t_2) &=&  (-1)^{n}f_{Y_n}(t_1,t_2);
   \\
   f_{Y_n}(t_1,t_2+\pi) &=&  (-1)^{\tfrac{n(n+1)}{2}}f_{Y_n}(t_1,t_2),
\end{eqnarray*}
while for cosine part we have
\begin{eqnarray*}
   \cos((t_1+\pi)\ell_1+t_2\ell_2) &=&  (-1)^{\ell_1} \cos(t_1\ell_1+t_2\ell_2) ;
   \\
    \cos(t_1\ell_1+(t_2+\pi) \ell_2) &=&  (-1)^{\ell_2}   \cos(t_1\ell_1+t_2\ell_2).
\end{eqnarray*}
Thefore, if $\ell \in D_n$, then the integrand in \eqref{pfcos} is $\pi$-periodical w.r.t.
both coordinates. In particular, it equals to 1 at any point of the set
$Q_n:=\{(0,0),(\pi,0),(0,\pi),(\pi,\pi)\}$.
On the other hand, it is trivial to see that $|f_{Y_n}(t)|<1$ for any $t\not\in Q_n$, $n\geq 2$.
\bigskip

What follows next is a ``three-domain approach'' of proving a CLT through characteristic functions: we divide the area of integration into three pieces: on the first piece the integrand is exponentially small, the second piece is not large enough to give any contribution, and the third piece, the integrand can be approximated by the corresponding normal characteristic function and thus gives the main contribution.

Let us start with the first piece. If we look at the integral in \eqref{pfcos} and consider the integration domain as a torus, then for large $n$
the integral is essentially accumulated in the small vicinity of the set $Q_n$. 

To be more precise, define the distance $d(t,s):=|t_1-s_1|+ n|t_2-s_2|$ and 
let $d(t,Q_n):=\inf_{s\in Q_n}d(t,s)$. One can show first that the integrand on the domain
$T_1:=\{ t:\, d(t,Q_n)\ge \eps\}$ is uniformly (in $t$ and $\ell$) exponentially small,
that is
\[
   \sup_{ t: \, d(t,Q_n)\ge \eps} \left|  f_{Y_n}(t) \right| \le (1-h(\eps))^n
\]
for some $h$ depending on $\eps$, see Lemma \ref{l:arithm}.
Therefore, everything reduces to the domain $\overline{T_1}:=\{ t: d(t,Q_n)< \eps\}$. By periodicity,
we can only consider $\{ t:\, d(t,0)\le \eps\}$= $\{ t:\, |t_1|+ n|t_2|\le \eps\}$ and
then multiply the result by $|Q_n|=4$.  

Next, the zone $\{ t:\, d(t,0)< \eps\}$ is further split into two zones,
$T_{2,M}:=\{ t:\, \tfrac{M}{\sqrt{n}}\le d(t,0)< \eps\}$ and 
$T_{3,M}:=\{ t:\, d(t,0) < \tfrac{M}{\sqrt{n}}\}$. On $T_{2,M}$ we use the bound
\begin{eqnarray*}
   |f_{Y_n}(t)| &=& \prod_{j=1}^n |\cos(t_1+jt_2)|
   \le \prod_{j=1}^n \exp\{-(t_1+jt_2)^2/2 \}
\\
   &=& \exp\left\{-\sum_{j=1}^n (t_1+jt_2)^2/2 \right\} 
   \leq 
   \exp\left\{-c (nt_1^2+ n^3 t_2^2) \right\},
\end{eqnarray*}
where we use  Lemma \ref{l:t1t2} in the last step. Therefore,
\begin{eqnarray*}
&& \lim_{M\to\infty} n^2 \int_{T_{2,M}} |f_{Y_n}(t)|\ dt 
\\
&\le& \lim_{M\to\infty} n^2 \int_{\{|t_1|+n|t_2|\ge \tfrac{M}{\sqrt{n}}\}} 
 \exp\left\{-c(nt_1^2+ n^3 t_2^2)\right\} dt 
 \\
& = &
 \lim_{M\to\infty} \int_{\{|s_1|+|s_2|\ge M\}} 
 \exp\left\{-c(s_1^2+ s_2^2)\right\} ds 
 =0,
\end{eqnarray*}
for any $n\in\N$, having used the change of variables $s_1=t_1\sqrt{n},s_2=t_2n^{3/2}$.

On $T_{3,M}$ by Taylor expansion we can compare $f_{Y_n}$ with the corresponding
normal characteristic function. Namely,
\begin{eqnarray*}
    f_{Y_n}(t) &=& \prod_{j=1}^n \cos(t_1+jt_2)
    = \exp\left\{\sum_{j=1}^n \ln\cos(t_1+jt_2)\right\}
 \\
    &=& \exp\left\{\sum_{j=1}^n \ln\left(1-(t_1+jt_2)^2/2+O(n^{-2})\right)\right\}
 \\
    &=& \exp\left\{- \tfrac 12 
        \left(\sum_{j=1}^n (t_1+jt_2)^2+O(n^{-1})\right)\right\}
\\
     &=& \exp\left\{  - \tfrac 12 \left( nt_1^2+ 2 \sum_{j=1}^n jt_1 t_2 
               + \sum_{j=1}^n j^2 t_2^2+O(n^{-1})\right)\right\}
\\
      &=& \exp\left\{  - \tfrac 12 \left( nt_1^2+ n^2 t_1 t_2 
               +  \frac{n^3}{3} t_2^2+O(n^{-1})\right)\right\}.
\end{eqnarray*}  
Therefore, using the change of variables $s_1=t_1\sqrt{n}, s_2=t_2n^{3/2}$ in the second step,
\begin{eqnarray*}
       &&\frac{n^2}{(2\pi)^2} \int_{T_{3,M}} f_{Y_n}(t) e^{-i(t_1\ell_1+t_2\ell_2)} dt_1dt_2 
\\
       &=& \frac{n^2}{(2\pi)^2} 
       \int_{\{|t_1|+n|t_2|\le \tfrac{M}{\sqrt{n}}\}} 
       \exp\left\{  - \frac 12 \left( nt_1^2+ n^2 t_1 t_2 
          +  \frac{n^3}{3} t_2^2 +O(n^{-1})\right)\right\}
       e^{-i(t_1\ell_1+t_2\ell_2)} dt_1dt_2 
\\
       &=&  \frac{1}{(2\pi)^2} \int_{\{|s_1|+|s_2|\le M \}} 
       \exp\left\{  - \frac 12 \left( s_1^2+  s_1 s_2 +  \tfrac 13 s_2^2 +O(n^{-1}) \right)\right\}
       e^{-i(s_1\frac{\ell_1}{\sqrt{n}} +s_2 \tfrac{\ell_2}{n^{3/2}})} ds_1ds_2 
\\
       &\to&   \frac{1}{(2\pi)^2} \int_{\{|s_1|+|s_2|\le M \}} 
       \exp\left\{  - \frac 12 \left( s_1^2+  s_1 s_2 +  \tfrac 13 s_2^2  \right)\right\}
       e^{-i(s_1 L_1 +t_2 L_2)} ds_1ds_2 
\end{eqnarray*}  
provided that $\ell_1=\left[L_1\sqrt{n}\right]$, $\ell_2= \left[L_2\, n^{3/2}\right]$,
and the convergence is uniform over $L_1,L_2$.
For large $M$ the latter limit is close,  uniformly over $L_1,L_2$, to
\begin{eqnarray*}
      && \frac{1}{(2\pi)^2} \int\int 
      \exp\left\{  - \frac 12 \left( s_1^2+  s_1 s_2 +  \tfrac 13 s_2^2  \right)\right\}
      e^{-i(s_1 L_1 +t_2 L_2)} ds_1ds_2
  \\
      &=& \frac {\sqrt{\det R}}{2\pi} \cdot   \frac {1}{2\pi \sqrt{\det R}  }  
      \int\int  \exp\left\{  - \frac 12 (R^{-1} s,s)\right\}
      e^{-i(s_1 L_1 +t_2 L_2)} ds_1ds_2
  \\  
      &=&  \frac {\sqrt{\det R}}{2\pi} \ \exp\left\{  - \frac 12 (R L,L)\right\},     
\end{eqnarray*}  
where
\[
   R^{-1} =\left(
   \begin{matrix}   1 & \tfrac 12
       \\ \tfrac 12& \tfrac 13 
    \end{matrix}   
       \right), \quad
   R=\left(
   \begin{matrix}  
       4 & -6\\ 
       -6& 12 
   \end{matrix} \right), \quad
   \det R=12.
\]
By recalling the factor $|Q_n|=4$, we arrive at
\begin{eqnarray*}
 \frac {4\sqrt{12}}{2\pi} \ \exp\left\{  - \frac 12 (R L,L)\right\}
 =
  \frac {4\sqrt{3}}{\pi} \ \exp\left\{  - 2L_1^2 +6 L_1L_2-6L_2^2\right\},
\end{eqnarray*} 
as required by the proposition.
$\Box$

\subsection{Some auxiliary lemmas}

\begin{lem} \label{l:t1t2} 
There exists $c>0$ such that for any $t_1,t_2\in \R$ and any integer $n\ge 2$ we have
\[
  \sum_{j=1}^n (t_1+jt_2)^2 \ge c(nt_1^2+ n^3 t_2^2).
\]  
\end{lem}

{\bf Proof.} \ There is no loss of generality to assume that $t_1=-1$ 
and $t:=t_2\ge 0$. Then we have to evaluate the function
\[
    G(t):= \sum_{j=1}^n (jt-1)^2 = S_2 t^2 -2S_1 t +n, 
\]
where $S_2=S_2(n)=\sum_{j=1}^n j^2\sim \tfrac{n^3}{3}$ and 
$S_1=S_1(n):=\sum_{j=1}^n j\sim \tfrac{n^2}{2}$.
The function $G(\cdot)$ attains its minimum at the point $t_*=\tfrac{S_1}{S_2}$ and
 \begin{equation} \label{nt1}
   \min_t G(t) =G(t_*) =n - \tfrac{S_1^2}{S_2} \ge (1-c_1) n,
 \end{equation}
 where $c_1:= \max_{n\ge 2} \tfrac{S_1^2}{S_2\, n}<1$.
 
 Notice that (using $S_1\leq n^2/2$)
 \[
 \frac{S_1}{S_2}\le \frac{c_1\, n}{S_1}\le \frac{2c_1}{n}.
 \]
 If $t\ge \tfrac{8c_1}{n}$, then 
 \[
   2S_1t   = \frac{S_2t^2}{2} \cdot \frac{S_1}{S_2} \cdot \frac{4}{t}\le 
   \frac{S_2t^2}{2} \cdot  \frac{2c_1}{n}\cdot \frac{4n}{8c_1} = \frac{S_2t^2}{2},
 \]
 hence, $G(t)\ge \frac{S_2t^2}{2}\ge \tfrac{n^3}{6}\, t^2$ (where we use $S_2\geq n^3/3$.
 Alternatively, if $0 \le t\le \tfrac{8c_1}{n}$, then
 \[
     G(t)\ge G(t_*)\ge (1-c_1) n = (1-c_1) n^3 t^2 \cdot \frac{1}{n^2t^2} 
     \ge \tfrac{1-c_1}{(8c_1)^2}\, n^3 t^2.
 \]
 Hence,
 \begin{equation} \label{n3t22}
    G(t) \ge \min\left\{\tfrac 16; \tfrac{1-c_1}{(8c_1)^2}  \right\}\, 
    n^3 t^2,    \qquad \forall t\in\R,
 \end{equation} 
 and the required assertion  $G(t)\ge c(n+n^3t^2)$  follows from \eqref{nt1} 
 and \eqref{n3t22}. $\Box$

\begin{lem} \label{l:arithm} For any $0<\eps<\pi/2$ 
there exists $h=h(\eps)\in (0,1)$ such that for any integer $n\ge 4$ we have
\[
   \sup_{ t: \, d(t,Q_n)\ge \eps} \left|  f_{Y_n}(t) \right| \le (1-h(\eps))^n.
\]
\end{lem}

{\bf Proof.} Let $M(x):=\min_{k\in\Z}|x-k\pi|$. Take any $t=(t_1,t_2)\in[0,\pi]^2$ such that
$d(t,Q_n)\ge \eps$. The latter means that
\[
  M(t_1)+nM(t_2)\ge\eps.
\]

Then two cases are possible:

1) $nM(t_2)\le \eps/2$.

Then we have $M(t_1) \ge \eps/2$. Hence, for any $1\le j \le n/2$ it is true that
\begin{eqnarray*}
    M(t_1+jt_2)&\ge& M(t_1) - M(jt_2)  \ge  M(t_1) - jM(t_2)
    \\
    &\ge& M(t_1) - nM(t_2)/2 \ge   \eps/2- \eps/4= \eps/4,
\end{eqnarray*}
and we obtain the required estimate
\[
     \left|  f_{Y_n}(t) \right| \le \prod_{j=1}^{n/2} |\cos(t_1+jt_2)|
     \le [\cos(\eps/4)]^{n/2}.
\]

2) $nM(t_2) \ge \eps/2$.

By symmetry reasons, there is no loss of generality in assuming  that 
$0<t_2\le \pi/2$.

Let $\dd:=\eps/10$. Then
\[
   \frac{nt_2}{\dd}=\frac{nM(t_2)}{\dd}\ge \frac{\eps/2}{\eps/10}=5.
\]
Choose an integer $m\ge 0$ such that $mt_2\le \delta\le (m+1)t_2$. 
Since $mt_2\le \delta\le  \tfrac{nt_2}{5}$, we have $m\le \tfrac{n}{5}$.

Assume for a while that $t_2\le\pi/3$.
We show now that for any $1\le j\le n-2(m+1)$ the inequalities
\[  
    M(t_1+jt_2)<\dd
\quad \textrm{and}\quad M(t_1+jt_2+ 2(m+1)t_2)<\dd 
\]
are incompatible. Indeed, let the first one be satisfied. Then for some $k\in\Z$ we have
\[
    \pi k -\dd < t_1+jt_2 < \pi k +\dd.
\]
It follows that
\begin{equation} \label{eqn:fasterunderst-}
   t_1+j t_2+ 2(m+1)t_2 > (\pi k -\dd)+2\dd = \pi k+\dd
\end{equation}
but (using twice that $\eps<1<\pi/2$)
\begin{eqnarray}
   t_1+j t_2+ 2(m+1)t_2 &<& (\pi k +\dd)+2(\dd+t_2) \notag
\\
    &=& \pi k+3\dd+ 2t_2 \le \pi k+3\pi/20 + 2\pi/3 \notag
\\   
   &<& \pi (k+1)-\pi/20 \le \pi (k+1)-\dd. \label{eqn:fasterunderst--}
\end{eqnarray}
From (\ref{eqn:fasterunderst-}) and (\ref{eqn:fasterunderst--}) it follows that 
$M(t_1+jt_2+ 2(m+1)t_2)>\dd$; and incompatibility is proved.
This fact yields
\[
   \#\{j\le n:\  M(t_1+jt_2)\ge \dd \} 
   \ge \frac{n-2(m+1)}{2}\ge \frac{3n}{10}-1, 
\] 
and
\[
     \left|  f_{Y_n}(t) \right| = \prod_{j=1}^{n} |\cos(t_1+jt_2)|
     \le [\cos(\dd)]^{3n/10-1},
\]
which settles the assertion of lemma.

For the remaining case  $\pi/3\le t_2\le\pi/2$ just observe that
the inequalities
\[  
 M(t_1+jt_2)<\dd
\quad \textrm{and}\quad M(t_1+(j+1)t_2)<\dd 
\]
are incompatible, and the proof goes along the same lines: we obtain
 \[
   \#\{j\le n:\  M(t_1+jt_2)\ge \dd \} \ge \frac{n-1}{2}, 
\] 
and
\[
     \left|  f_{Y_n}(t) \right| = \prod_{j=1}^{n} |\cos(t_1+jt_2)|
     \le [ \cos(\dd)]^{(n-1)/2}. \qquad \Box
\]

\section{Scaling of $S$ and $A$ under the conditioning} \label{sec:scaling}

The purpose of this section is to show the following facts concerning the scaling 
of $S$ and $A$ under the condition of positivity of $A$.

\begin{prop}
There is a constant $c>0$ such that, for all $n\geq 1$,
 \begin{eqnarray}
 \E\left( |S_n|  ~|~ \Omega_n^+\right) & \leq& c n^{1/2},  \label{eqn:snresult}\\
 \E\left( A_n ~|~ \Omega_n^+\right) & \leq& c n^{3/2}.    \label{eqn:anresult}
 \end{eqnarray} \label{prop:scaling}
\end{prop}

\begin{proof} 1.) We start with the proof of $(\ref{eqn:snresult})$.A twe first we show that it suffices to estimate $\E(S_n^+\ |\ \Omega_n^+)$, since
\begin{equation} \label{eqn:normtopositivepart}
\E\left( |S_n| | \Omega_n^+ \right)
  =    \E\left( S_n^+ + S_n^- | \Omega_n^+ \right)
 \leq 2 \E\left( S_n^+ | \Omega_n^+ \right).
\end{equation}
Indeed, if $S_n<0$ we let $\sigma_0$ denote the last visit to zero: $\sigma_0:=\max\{ k\leq n : S_k=0\}$. Consider the path transformation that is given by inverting the steps after $\sigma_0$. This transformation maps any path in $\Omega_n^+$ with $S_n<0$ to a path in $\Omega_n^+$ with $S_n>0$. It is a one-to-one transformation (however, not a bijection, since the image of a transformation of a path in $\Omega_n^+$ with $S_n>0$ does not have to be in $\Omega_n^+$). Therefore,
 $$ 
\pr{ \Omega_n^+\cap \{ S_n = -k\}} \leq \pr{ \Omega_n^+\cap \{ S_n = k\}},\qquad k>0,
$$ 
and thus
$$
\E\left( S_n^- | \Omega_n^+ \right) \leq \E\left( S_n^+ | \Omega_n^+ \right).
$$

It remains to estimate the latter expectation.
For $0\leq t\leq n$, let
\begin{equation} \label{eqn:*last}
\Omega_{n,t}:=\left\lbrace A_j\ge 0, 1\le j\le t;\ S_t=0;\ S_j >0, t+1\le j \le n \right\rbrace.
\end{equation}
Clearly,
$$
\Omega_n^+= (\Omega_n^+\cap\{ S_n<0 \}) \cup \bigcup_{t=0}^n \Omega_{n,t}.
$$
Thus,
 \begin{eqnarray}
  \E\left( S_n^+ | \Omega_n^+ \right) &=& \frac{1}{\pr{\Omega_n^+}} \, \E S_n^+ \sum_{t=0}^n \ind_{\Omega_{n,t}}\notag \\
 &=& \sum_{t=0}^n \frac{\E S_n^+  \ind_{\Omega_{n,t}}}{\pr{\Omega_{n,t}}} \, \frac{\pr{\Omega_{n,t}}}{\pr{\Omega_n^+}}\notag  \\
 &\leq &  \max_{0\leq t\leq n} \E \left(S_n^+ | \Omega_{n,t}\right) \sum_{t=0}^n \frac{\pr{\Omega_{n,t}}}{\pr{\Omega_n^+}} \notag \\
 &\leq &  \max_{0\leq t\leq n} \E \left(S_n^+ | \Omega_{n,t}\right). \label{eqn:twost1111}
 \end{eqnarray}

By definition,
\[
  \E\left(S_n^+| \Omega_{n,t}\right)=\sum_{k>0} 
  \frac{\P\left(\Omega_{n,t}\cap\{S_n=k\}\right)k}{\P\left(\Omega_{n,t}\right)}.
\]
We can represent each of events here as an intersection of two independent events, respectively:
\begin{eqnarray*}
\Omega_{n,t}&=& \Omega_t^+\cap\left\lbrace S_t=0\right\rbrace \bigcap \left\lbrace S_j-S_t >0, t+1\le j \le n \right\rbrace;
\\
\Omega_{n,t}\cap\lbrace S_n=k\rbrace &=& \Omega_t^+\cap\left\lbrace S_t=0\right\rbrace \bigcap \left\lbrace S_j-S_t >0, t+1\le j \le n;\ S_n- S_t=k \right\rbrace.
\end{eqnarray*}
It follows that
\begin{eqnarray}
 && \E\left(S_n^+| \Omega_{n,t}\right) \notag
 \\
 &=&\sum_{k>0} 
 \frac
 {\P\left(\Omega_t^+\cap\{S_t=0\}\right)
  \P\left( S_j-S_t >0, t+1\le j \le n; \ S_n- S_t=k \right)k}
 {\P\left(\Omega_t^+\cap\{S_t=0\} \right)
  \P\left( S_j-S_t >0, t+1\le j \le n \right)} \notag
\\
&=&\sum_{k>0} 
 \frac
 {\P\left( S_j-S_t >0, t+1\le j \le n; \ S_n- S_t=k \right)k}
 {\P\left( S_j-S_t >0, t+1\le j \le n \right)} \notag
\\
&=&\sum_{k>0} 
 \frac
 {\P\left( S_i >0, 1\le i \le n-t;\  S_{n-t}=k\right)k}
 {\P\left( S_i >0, 1\le i \le n-t \right)} \notag
\\
&=& \E\left(S_{n-t}^+| S_i >0, 1\le i \le n-t \right),\qquad 0\leq t\leq n. \label{eqn:onest1111}
\end{eqnarray}

In order to evaluate the latter expectation we use a stopping time argument.
Let $v:=\inf\{k: S_k=-1\}$ and $v_n:=\min(v,n)$. Then
$v_n$ is a bounded stopping time and we have
\[
0=\E S_{v_n}= \E S_n \ind_{v>n}-\P(v\le n).
\]
Hence, $ \E S_n \ind_{v>n}=\P(v\le n)\le 1$. 

On the other hand, we know (see \cite{feller2}, XII.8) that $\P(v>n)\approx n^{-1/2}$.

Finally, let us consider $S'_i:=S_{i+1}-S_1$, $0\le i\le n,$ and let $v'$ be 
the corresponding stopping time. Then
\begin{eqnarray}
   \E\left(S_{n+1}^+| S_i >0, 1\le i \le n+1 \right) &=& 1 + \E\left(S'_{n}| S'_i \ge 0, 1\le i \le n \right)
   \notag \\
   &=& 1+  \E\left(S'_n | v'>n \right)
   \notag  \\
   &=&   1+ \frac{\E S'_n \ind_{v'>n}}{\P(v'>n)} 
   \notag  \\
   &\le& 1 + \tfrac{\sqrt{n}}{c}  \le  C' \sqrt{n}. \label{eqn:squarealso}
\end{eqnarray}

Combining this with (\ref{eqn:twost1111}) and (\ref{eqn:onest1111}) gives
\begin{equation} \label{eqn:splusresult}
 \E \left( S_n^+ | \Omega_n^+ \right) \leq c \sqrt{n}.
\end{equation}
This and (\ref{eqn:normtopositivepart}) show (\ref{eqn:snresult}).

2.) We now prove $(\ref{eqn:anresult})$. We start with some simple estimates:
\[
  |A_n|=\left|\sum_{k\le n} S_k\right| \le \sum_{k\le n} \left|S_k\right|
  = \sum_{k\le n} (S_k^+ + S_k^-)= \sum_{k\le n} S_k^+  + \sum_{k\le n} S_k^-.
\]
Moreover, on $\Omega_n^+$ we have
\[
  0\le A_n=\sum_{k\le n} S_k = \sum_{k\le n} (S_k^+ - S_k^-)
  = \sum_{k\le n} S_k^+  - \sum_{k\le n} S_k^-.
\]
Hence,
\[
   |A_n|\le 2 \sum_{k\le n} S_k^+ \le 2n \max_{0\le k\le n} S_k.
\]
It is now enough to prove that for any $R\in\N$ it is true that
\begin{equation} \label{key}
  \P\left( \left\lbrace \max_{0\le k\le n} S_k\ge R\right\rbrace\cap \Omega_n^+\right)
  \le 2 \  \P\left( \left\lbrace S_n\ge R\right\rbrace\cap \Omega_n^+\right),
\end{equation}
because this leads to the desired
\[
   \E\left(|A_n|\big| \Omega_n^+\right)\le 2n \E\left(\max_{0\le k\le n} S_k\big| \Omega_n^+\right)
   \le 4n \E\left(S_n^+\big| \Omega_n^+\right)\le 4n\cdot c \sqrt{n}= c' n^{3/2},
\]
where we used (\ref{eqn:splusresult}) in the third step.

For proving \eqref{key} we will only use the monotonicity property of $\Omega_n^+$: if
$x$ and $y$ are two paths with $x\in \Omega_n^+$ and $y\ge x$ pointwise then $y\in \Omega_n^+$.

Fix $R\in\N$. For any $S\in \left\lbrace \max_{0\le k\le n} S_k\ge R\right\rbrace\cap \Omega_n^+ $ the time
\[
   \sigma_R:=\max\{k\leq n: S_k=R\}\le n
\]
is well defined.
For any $S\in \left\lbrace \max_{0\le k\le n} S_k\ge R\right\rbrace\cap \Omega_n^+ \cap \lbrace S_n< R \rbrace$
define its transformation $y$ by inverting steps starting from $\sigma_R$.
This is one-to-one transformation and we have the following properties of $y$:
$S_k=y_k$ for $k\le \sigma_R$ while $S_k<R<y_k$ for $\sigma_R<k\le n$ (in particular, $y_n>R$).
 Hence $S\le y$ pointwise. By monotonicity, $y\in \Omega_n^+ $.

We infer that our transformation is a one-to-one embedding (as a side note: it is not a bijection, since the image of the transformation of a path in  $\Omega_n^+ \cap \{ S_n> R \}$ may be outside $\Omega_n^+$):
\[
\left\lbrace \max_{0\le k\le n} S_k\ge R\right\rbrace\cap \Omega_n^+ \cap \lbrace S_n< R\rbrace
\rightarrow   \Omega_n^+ \cap \lbrace S_n> R\rbrace.
\]
Hence, 
\[
\P\left(\left\lbrace \max_{0\le k\le n} S_k\ge R\right\rbrace\cap \Omega_n^+ \cap \lbrace S_n< R\rbrace\right)
\le
\P\left( \Omega_n^+ \cap \lbrace S_n> R\rbrace\right).
\]
and \eqref{key} follows.

\end{proof}

The following lemma is also concerned with the scaling of $S$ and $A$. 
We show that the joint distribution of $S_n$ or $A_n$ (conditioned on $\Omega_n^+$) 
is not concentrated on negative values or near zero when $n\to\infty$.

\begin{lem} \label{lem:spositivecon}
For $l,m,n\in\N$ with $l<n$ one has
$$
\P(S_{n}\geq m,A_{n}\geq l m~|~\Omega_{n}^+)\geq \P(S_l\geq 2m) \,\P( |S_{n-l}|\leq  m)
$$
In particular, for any constants $c_1, c_2>0$, there exists a strictly positive constant $\kappa=\kappa(c_1,c_2)$ such that for all sufficiently large $n$
$$
\P(S_{n}\geq c_1 n^{1/2} ,A_n\geq c_2 n^{3/2}  ~|~\Omega_n^+)\geq \kappa >0.
$$
\end{lem}

\begin{proof}
First note that
\begin{align}
& \P(S_{n}\geq m,\,  A_{n} \geq l m~|~\Omega_n^+) \notag \\
&\geq \P(S_{l}\geq 2m, S_i-S_l\geq -m  \text{ for }i=l+1,\dots,n ~|~\Omega_{n}^+) \notag \\
&\geq \frac{\P(\{ S_{l}\geq 2m\}\cap \Omega_{l}^+ \cap \{ S_i-S_l\geq -m  \text{ for }i=l+1,\dots,n\})}{\pr{\Omega_{n}^+}}, \label{eqn:contre}
\end{align}
since for $k=l+1,\ldots,n$
\begin{align*}
 A_k =    \, & A_l  + \sum_{i=l+1}^k (S_i-S_l) + (k-l) S_l\\
     \geq \, & 0 + \sum_{i=l+1}^k (-m) + (k-l) S_l \\
     = \, & (k-l) ( S_l-m)\geq (k-l) m \geq 0.
\end{align*}
By independence, the term in (\ref{eqn:contre}) equals
\begin{align*}
&=\frac{ \P(\{S_l\geq 2m\}\cap \Omega_{l}^+) \,\P(\min_{i=1,\dots,n-l} S_i\geq -m)}{\P(\Omega_{n}^+)}\\
&\geq \P(S_l\geq 2m\, |\, \Omega_{l}^+) \,\P(\max_{i=1,\dots,n-l} S_i\leq m)\\
&\geq  \P(S_l\geq 2m) \,\P(|S_{n-l}|\leq  m),
\end{align*}
where, in the last step, we used the reflection principle as well as the fact that
\begin{equation} \label{eqn:associated}
\P(\{ S_l\geq 2m\} \cap \Omega_{l}^+) \geq \P(S_l\geq 2m) \cdot \P(\Omega_{l}^+),
\end{equation}
which means that the events are positively correlated: 
Recall that a family of random variables $(X_i)_{1\le i\le l}$ is called {\it associated} 
if for any pair of bounded coordinate-wise non-decreasing functions $f_1,f_2:\R^l\to \R^1$
it is true that
\[
   \E \left(f_1(X_1,\dots,X_l) f_2(X_1,\dots,X_l)\right)\ge 0.
\] 
See \cite{BS} for detailed account of the association property and its extensions. One only needs to know that
any family of independent random variables is associated, cf.\ Theorem 1.8 in \cite{BS} due to \cite{esaryetal}. Thus, (\ref{eqn:associated}) holds.
\end{proof}

\section{CLT for the pinned process} \label{sec:cltpinned}
In this section we prove a central limit theorem for the pinned process  $(S_n,A_n)_{n=1,\dots,N}$, when letting $N\in2\N$ tend to infinity. By ``pinning'' we mean that the process is conditioned to arrive at a certain point, depending on $N$ and scaling in $N$ with the natural scaling of the process. We restrict attention to even numbers $N$ for technical reasons, although the following theorem remais valid for general $N$. \smallskip

We need some more notation. Recall that the set $D_n$ was defined in (\ref{Dn}). 
We describe the pinning via an $\R^2$-valued  sequence $(\pin^{\ssup N})_{N\in2\N}$ satisfying
$$
  (N^{1/2} \pin^{\ssup N}_1, N^{3/2} \pin^{\ssup N}_2)\in D_N \ \text{ and } \ \lim_{N\to\infty} \pin^\ssup{N} = \pin
$$
for a $\pin \in\R^2$. Further, as described in the introduction, we associate 
to $(S_n,A_n)_{n=0,\dots,N}$ the adjoint process 
$(\bar S^\ssup{N}_n,\bar A^\ssup{N}_n)_{n=0,\dots,N}$ started at
$$
  \bar S^\ssup{N}_0= N^{1/2} \pin_1^\ssup{N} \text{ and }  \bar A^\ssup{N}_0= N^{3/2} \pin_2^\ssup{N}.
$$

The original process $(S_n,A_n)_{n=0,\dots,N}$ and its adjoint process  
$(\bar S^\ssup{N}_n,\bar A^\ssup{N}_n)_{n=0,\dots,N}$ will be considered in their 
normalized versions: we set for $N\in2\N$ and $s\in \frac 1N \Z \cap[0,1]$
$$
     \Xi^{\ssup N}_s := ( N^{-1/2} S_{Ns} ,  N^{-3/2} A_{Ns}^\ssup{N} ) 
     \ \text{ and } 
     \ \bar \Xi^{\ssup N}_s := (N^{-1/2} \bar S_{Ns} ,N^{-3/2}\bar  A_{Ns}^\ssup{N}) 
$$
and apply a continuous piecewise linear interpolaton between the breakpoints $s\in \frac 1N \Z \cap[0,1]$ to obtain  continuous processes $\Xi^\ssup N=(\Xi^\ssup{N}_s)_{s\in[0,1]}$ and $\bar \Xi^\ssup N=(\bar \Xi^\ssup{N}_s)_{s\in[0,1]}$.

Using this notation, the CLT reads as follows.

\begin{thm} \label{pinned_CLT} One has
$$
    \cL(\Xi^\ssup{N}\,|\, \Xi^\ssup{N}_1= \pin^\ssup{N}) \Rightarrow \cL(\Gamma | \Gamma_1=\pin),
$$
where $\Gamma= (B_t,I_t)_{t\in[0,1]}$.
\end{thm}

\begin{rem}\label{rem1}We remark that the theorem remains valid when choosing different starting points for the  Markov process $(S_n,A_n)_{n\in\N}$. Suppose that it is started in $(s^{\ssup N}, a^{\ssup N})$ such that the limit
$$
\mathfrak s= \lim_{N\to\infty} (N^{-1/2} s^{\ssup N}, N^{-3/2} a^{\ssup N})
$$
exists. Now assuming that the pinning is done on non null events, one gets
$$
\cL^{(s^{\ssup N},a^{\ssup N})}(\Xi^{\ssup N}| \Xi_1^{\ssup N}) \Rightarrow \cL^{\mathfrak s}(\Gamma|\Gamma_1=\mathfrak p).
$$
Here the right hand side denotes the law of integrated Brownian motion started in $\mathfrak s$.  The statement is straight-forwardly obtained by using that  
$$
(\tilde S_n,\tilde A_n):= (S_n+ s,A_n+a+ns)_{n\in \N}
$$
has under $\P^{(0,0)}$ the same distribution as $(S_n,A_n)$ under $\P^{(s,a)}$, and the analogous property for the process $\Gamma$.
\end{rem}

In order to prove Theorem~\ref{pinned_CLT}, we prove tightness and convergence of finite-dimensional distributions for the conditioned distributions.


\begin{proof}[Proof of tightness in Theorem \ref{pinned_CLT}] We prove that the sequence of conditional distributions $\cL(\Xi^\ssup{N}\,|\, \Xi^\ssup{N}_1= \pin^\ssup{N})$
on $C([0,1],\R^2)$ is tight.

Let $\eps>0$ and fix compact sets $K_1,K_2$ in $C([0,\frac 12],\R^2)$ with 
$$
\P(\Xi^{\ssup N}  \in K_1) \geq 1-\eps \ \text{ and } \ \P(\bar \Xi^\ssup{N}  \in K_2)\geq 1-\eps
$$
for all $N\in2\N$, where the processes are to be considered on the time interval $[0,\frac 12]$. 
Such compact sets exist by Donsker's invariance principle (see, e.g., \cite{Bil99}).
 Now  let $K\subset C[0,1]$ be the set of continuous functions $f:[0,1]\to \R^2$ with
$$
(f(t):t \in[0,\sfrac12]) \in K_1 \ \text{ and } \ (f(1-t):t\in [0,\sfrac 12])\in K_2.
$$
It is obviously compact in $C([0,1],\R^2)$. 
By the definition of the adjoint process (see~(\ref{eqn:adjointprop})), one has
$$
\{\Xi^\ssup{N}_1=\pin^\ssup{N}\} = \{\Xi^\ssup{N}_{1/2} = \bar \Xi^\ssup{N}_{1/2} \} = \{ \bar \Xi^\ssup{N}_1=0\}
$$
so that
$$
\P(\Xi^\ssup{N} \not \in K| \Xi_1^\ssup{N}=\pin^{\ssup N})\leq \P(\Xi^\ssup{N} \not \in K_1| \Xi^\ssup{N}_{1/2} = \bar \Xi^\ssup{N}_{1/2})+ \P(\bar \Xi^\ssup{N} \not \in K_2|\Xi^\ssup{N}_{1/2} = \bar \Xi^\ssup{N}_{1/2})
$$
To obtain an upper bound for the first term, we observe that $(\ind\{\Xi^\ssup{N} \not \in K_1\}, \Xi^\ssup{N}_{1/2})$ and $\bar \Xi^\ssup{N}_{1/2}$ are independent which implies that
$$
\P(\Xi^\ssup{N} \not \in K_1, \Xi_{1/2}^\ssup{N}=\bar \Xi_{1/2}^\ssup{N})  = \sum_{z} \P(\Xi^\ssup{N}\not \in K_1, \Xi^\ssup{N}_{1/2}=z) \, \P(\bar  \Xi^\ssup{N}_{1/2}=z)
$$
By the local central limit theorem, the weights $\P(\bar  \Xi^\ssup{N}_{1/2}=z)$ are uniformly bounded by a constant multiple of $N^{-2}$ so that
$$\P(\Xi^\ssup{N} \not \in K_1, \Xi_{1/2}^\ssup{N}=\bar \Xi_{1/2}^\ssup{N}) \leq   C_1 \frac 1{N^2}  \P(\Xi^{\ssup N}\not \in K_1) \leq C_1 \eps \frac 1{N^2}
$$
for a universal constant $C_1$. Analogously, one concludes that
$$\P(\bar \Xi^\ssup{N} \not \in K_2, \bar \Xi_{1/2}^\ssup{N}= \Xi_{1/2}^\ssup{N}) \leq   C_2 \frac 1{N^2}  \P(\bar \Xi^\ssup{N} \not \in K_1) \leq C_2 \eps \frac 1{N^2}
$$
for a universal contant $C_2$. Since by the local central limit theorem (Proposition~\ref{prop:llt} above) $\lim_{N\to\infty} N^2\,\P(\Xi^\ssup{N}_1=\pin^\ssup{N})=C_3>0$, we conclude that
$$
\limsup_{n\to\infty} \P(\Xi^\ssup{N} \not \in K| \Xi_1^{\ssup N}=\pin^{\ssup N}) \leq \frac {C_1+C_2}{C_3} \,\eps.
$$
Since $\eps>0$ was arbitrary, this proves tightness.
\end{proof}

\begin{proof}[Proof of convergence of finite-dimensional distributions in Theorem \ref{pinned_CLT}] It remains to prove convergence of finite dimensional marginals.   For  $t>0$ we denote by $g_t:\R^2 \times \R^2\to [0,\infty) $ the transition density of the Markov process $(\Gamma_s)$ over an interval of length~$t$. It is
$$
g_t(u,v;x,y)= t^{-2}\,g\Bigl(\frac{x-u}{\sqrt t},\frac {y-v-tu}{t^{3/2}}\Bigr).
$$
 
Fix $m\in\N$, times $0<t_1<\dots <t_m<1$ and a continuous and bounded  function $f:(\R^2)^m \to [0,\infty)$. We denote $t_m^*=t_m^*(N)= \min (\Z/N)\cap [t_m,1]$.\smallskip

First we verify that for arbitrary $\eps>0$, one has for sufficiently large $N\in2\N$ that 
\begin{align}\label{eq0102-2}
\E[\ind_{\{\Xi_1^\ssup{N} = \pin^{\ssup N}\}} | \Xi^{\ssup N}_{t_m^*}=z ] \geq (4 g_{1-t_m^*}(z,\pin)- \eps)N^{-2}
\end{align}
for all $z$ with $\P( \Xi_{t_m^*}^\ssup{N}=z)>0$. The estimate is a consequence of the local central limit theorem: we let $n=n(N)=(1-t_m^*)N$, $\ell=\ell(N)=(\sqrt N \pin^{\ssup N}_1,N^{3/2} \pin^{\ssup N}_2)$ and $\zeta=(\sqrt N z_1,N^{3/2} z_2)$, where $z$ is as before. For any $\eps'>0$, one has uniformly in the relevant $z$'s that for sufficiently large $N$ one has
\begin{align}\begin{split}\label{eq0102-1}
\P(\Xi_1^\ssup{N} = \pin^{\ssup N} | \Xi^{\ssup N}_{t_m^*}=z) &= \P^{\zeta}((S_n,A_n)=\ell) \\
&=\P((S_n,A_n)=(\ell_1-\zeta_1, \ell_2-\zeta_2-n \zeta_1))\\
&\geq 4 n^{-2}  g\Bigl(\frac {\ell_1-\zeta_1}{\sqrt n}, \frac {\ell_2-\zeta_2-n \zeta_1}{n^{3/2}}\Bigr)-\eps' n^{-2}\\
&= 4 N^{-2}  g_{1-t_m^*}\bigl(z_1,z_2;\pin_1^\ssup{N},\pin_2^\ssup{N}\bigr)-\eps' n^{-2}.
\end{split}\end{align}
Since $n(N)$ is of order $N$, we can choose for given $\eps>0$ a sufficiently small $\eps'>0$ such that~(\ref{eq0102-1}) implies ~(\ref{eq0102-2}).

Hence, by the Markov property one has
\begin{align*}
 & \E[f(\Xi_{t_1}^\ssup{N},\dots, \Xi_{t_m}^\ssup{N}) \, \ind_{\{\Xi_1 ^\ssup{N}= \pin^{\ssup N}\}}]\\
 & = \E[f(\Xi_{t_1}^\ssup{N},\dots, \Xi_{t_m}^\ssup{N})  \, \E[\ind_{\{\Xi_1 ^\ssup{N}= \pin^{\ssup N}\}}|\cF_{t_{m}^*}]]\\
 &\geq 4 N^{-2} \, \E[f(\Xi_{t_1}^\ssup{N},\dots, \Xi_{t_m}^\ssup{N}) \, g_{1-t_m^*}(\Xi_{t_m^*} ^\ssup{N},\pin) ] -\eps C\, N^{-2},
\end{align*}
where $C$ denotes a universal bound for $f$. Using the continuity of $g$ together with the classical Donsker invariance principle \cite{Bil99}, we arrive at
\begin{align*}
 \liminf_{n\to\infty} &\sfrac 14 N^2 \E[f(\Xi_{t_1}^\ssup{N},\dots, \Xi_{t_l}^\ssup{N}) \, \ind_{\{\Xi_1^\ssup{N} = \pin^{\ssup N}\}}]\\
& \ \ \geq \E[f(\Gamma_{t_1},\dots, \Gamma_{t_m}) \, g_{1-t_m}(\Gamma_{t_m},\pin) ] -\eps C.
\end{align*}
Analogously, one proves the converse bound. Since $\eps>0$ is arbitrary, we get that
\begin{align*}
\lim_{n\to\infty} \sfrac14 N^2\, &\E[f(\Xi_{t_1}^\ssup{N} ,\dots, \Xi_{t_m}^{\ssup N} ) \, \ind_{\{\Xi^\ssup{N}_1 = \pin^{\ssup N}\}}] =\E[f(\Gamma_{t_1},\dots, \Gamma_{t_m}) \, g_{1-t_m}(\Gamma_{t_m},\pin) ] 
\\
& \ \ = \int\dots \int f(z_1,\dots,z_l)\, g_{t_1}(0, z_1) \dots g_{1-t_m}(z_m,\pin) \,d z_1 \dots d z_m
\end{align*}
Conversely,
$$
\lim_{n\to\infty} \sfrac14 N^2\,  \E[\ind_{\{\Xi_1^{\ssup N} = \pin^{\ssup N}\}}] = g_1(0,\pin).
$$
\end{proof}

\section{Proof of the main theorem} \label{sec:proofmainthm}
\subsection{Proof of the upper bound}
The purpose of this section is to prove the upper bound in (\ref{eqn:mainht2}). This will follow almost directly from a local limit theorem for $(S_n,A_n)$ (Proposition~\ref{prop:llt}).

Let us recall that
$$
\Omega^+_n=\{ A_1\geq 0, \ldots, A_{n}\geq 0\} \in \sigma(X_1,\ldots, X_n)
$$
and let us define
\begin{align*}
\bar\Omega^+_n
               & =\{ \bar A^\ssup{4n}_1\geq 0, \ldots, \bar A^\ssup{4n}_{n}\geq 0\}
                \in \sigma(X_{3n+1},\ldots, X_{4n}),
\end{align*}
where the adjoint process is started at $(0,0)$. Then due to the fact that $\bar A^\ssup{4n}_k=A_{4n-k}$ on $A_{4n}=S_{4n}=0$ (see (\ref{eqn:adjointprop}) we have
\begin{align*}
& \pr{ A_1\geq 0, \ldots, A_{4n}\geq 0 , A_{4n}=S_{4n}=0 } \\
\leq~& \pr{ \Omega^+_n\cap \bar\Omega^+_n\cap\{ A_{4n}=S_{4n}=0 \} }\\
 = ~ & \pr{ A_{4n}=S_{4n}=0 \left| \Omega^+_n\cap \bar\Omega^+_n \right.} \cdot \pr{\Omega^+_n\cap\bar\Omega^+_n}.
\end{align*}
 Clearly, $\Omega^+_n$ and $\bar\Omega^+_n$ are independent. 
Further, the adjoint process started at $(0,0)$ has the same distribution as the original process, so that $\pr{\Omega^+_n}=\pr{\bar\Omega^+_n}\approx n^{-1/4}$, by Sinai's result \cite{sinai1991}.

On the other hand, the event
$\Omega^+_n\cap\bar\Omega^+_n$ only concerns the random variables $X_1,\ldots, X_n$ and $X_{3n+1},\ldots, X_{4n}$, so that
\begin{align} 
& \pr{ A_{4n}=S_{4n}=0 \left| \Omega^+_n\cap\bar\Omega^+_n\right.} \notag \\   \leq~ & \sup_{\substack{x_1,\ldots, x_n,\\ x_{3n+1},\ldots, x_{4n}\in\{-1,+1\}}} \pr{ A_{4n}=S_{4n}=0 \left|\substack{X_1=x_1,\ldots, X_n=x_n,\\ X_{3n+1}=x_{3n+1},\ldots, X_{4n}=x_{4n}} \right.} \label{eqn:flk--}
\end{align} 
Given the values for $X_1,\ldots, X_n$ and $X_{3n+1},\ldots, X_{4n}$, the variables $(S_i,A_i)$, $i\in\{ n+1,\ldots, 3n \}$, form another pair of simple random walk and its integrated counterpart, however, the pair is started at some different point: i.e.\ for the vector $(S_i,A_i)$, $i=n+1,\ldots,2n$, started at $(k,l)$ has the distribution
$$
(k+S_{i-n},k(i-n-1)+l+A_{i-n}).
$$
Thus, the quantity in (\ref{eqn:flk--}) equals
$$
 \sup_{a,s,a',s'} \P^{(s,a)}\left(  A_{2n}=a',S_{2n}=s'\right) = \sup_{a',s'} \P^{(0,0)}\left(  A_{2n}=a',S_{2n}=s'\right).
$$
Now, the local limit theorem for $(S_n,A_n)$ (Proposition~\ref{prop:llt}) tells us that the latter probability is bounded by $c n^{-2}$.

Summing up and using once again the local limit theorem, we get
\begin{align*}
& \pr{ A_1\geq 0, \ldots, A_{4n}\geq 0 \left| A_{4n}=S_{4n}=0 \right. } \\ 
=~ & \frac{\pr{ A_1\geq 0, \ldots, A_{4n}\geq 0 , A_{4n}=S_{4n}=0 }}{\pr{ A_{4n}=S_{4n}=0 }} \\
\leq~  & {\rm const.}\, \frac{n^{-2} \cdot n^{-1/4} \cdot n^{-1/4}}{n^{-2}}\\
 \approx~  & n^{-1/2},
\end{align*}
which finishes the proof of the upper bound.

\subsection{Proof of the lower bound}
Here we prove the lower bound in (\ref{eqn:mainht2}).

First observe that 
Proposition~\ref{prop:scaling} and Lemma~\ref{lem:spositivecon} show that there are constants $0<a<b<\infty$ and $\kappa>0$ such that, for all $n$,
\begin{equation} \label{eqn:nodr+}
\P\left(  S_{n} \in [a n^{1/2},b n^{1/2}], A_{n}\in [a n^{3/2},b n^{3/2}] ~ |~ \Omega_{n}^+ \right) \geq \kappa >0.
\end{equation}

Using the Markov property of $(S_n,A_n)$ we have
\begin{align}
 \P( \Omega_{4n}^+ \cap& \{  (S_{4n},A_{4n})=(0,0) \} ) \notag\\
 =\, & \sum_{k,l} \sum_{k',l'} \P^{(0,0)}( \Omega_{n}^+\cap \{  (S_{n},A_{n})=(k,l) \} ) \notag \\ 
 & ~ ~~ ~ ~~ ~~~ ~\cdot \P^{(k,l)}( \Omega_{2n}^+ \cap \{  (S_{2n},A_{n})=(k',l') \} )   \notag \\
  & ~ ~~ ~ ~~ ~~~ ~\cdot \P^{(k',l')}( \Omega_{n}^+ \cap \{  (S_{n},A_{n})=(0,0) \} ). \label{eqn:lwmainarg}
\
 \end{align}
We denote by $(\bar S_k^\ssup{n},\bar A_k^\ssup{n})_{k=0,\ldots,n}$ denote the adjoint process started in $(0,0)$ and consider 
$$
\bar\Omega_n^+=\{\bar A_k^\ssup {n}\geq0\text{ for all }k=0,\dots,n\}.
$$
By (\ref{eqn:adjointprop}), one has
\begin{align*}
 \P^{(k',l')} ( \Omega_n^+ \cap \{  (S_{n},A_{n})=(0,0)\}  ) & = \P^{(k',l')}( \bar\Omega_n^+ \cap \{ (\bar S_n^\ssup{n},\bar A_n^\ssup{n})=(k',l') \} ) \\
& = \P^{(0,0)}( \Omega_n^+ \cap \{  (S_{n},A_{n-1})=(-k',l') \} ),
\end{align*}
where we used in the last step that $((\bar S^{\ssup n}_m, \bar A^{\ssup n}_m: m=0,\dots,n)$  under $\P^{(k',l')}$ is identically distributed as  $((-S_{m},A_{m-1}):m=0,\dots,n)$ under $\P^{(0,0)}$. Here we adopt the convention $A_{-1}= A_0-S_0$.
In order to compute  a lower bound for (\ref{eqn:lwmainarg}),  we  can confine ourselves to summands where $(k,l)$ and $(-k',l')$ are in $[a {n}^{1/2}, b{n}^{1/2}]\times [a {n}^{3/2}, b {n}^{3/2}]$, respectively, that is to
$$
\cL(n)=\{ (k,l)\in D_n ~|~ (k,l)\in [a {n}^{1/2}, b{n}^{1/2}]\times [a {n}^{3/2}, b {n}^{3/2}]\}
$$
and
$$
\mathcal R(n)=\{ (k',l')\in \tilde D_n ~|~ (-k',l')\in [a {n}^{1/2}, b{n}^{1/2}]\times [a {n}^{3/2}, b {n}^{3/2}]\},
$$
respectively, where $D_n$ is defined in (\ref{Dn}) and
$$
 \tilde D_n:=\{\ell=(\ell_1,\ell_2)\in \Z^2:\ \ell_1=n\mod 2,\quad \ell_2=\tfrac{(n-1)n)}{2}\mod 2 \}.
$$
We will prove below that
\begin{equation} \label{eqn:nodr-}
2\kappa':= \liminf_{n\to\infty} n^2 \,\inftwo{(k,l) \in \mathbb \cL(n)}{(k',l')\in \mathcal R(n)} \P^{(k,l)}( \Omega_{2n}^+ \cap \{  (S_{2n},A_{2n})=(k',l') \} ) >0.
\end{equation}
Then one estimates (\ref{eqn:lwmainarg})  by the product of $\kappa' n^{-2}$,
$$
 \sum_{(k,l)\in\mathcal{L}(n)} \P( \Omega_{n}^+\cap \{  (S_{n},A_{n})=(k,l) \} ) \geq  \pr{\Omega_{n}^+} \kappa,
$$
and
$$
 \sum_{(k',l')\in\mathcal{R}(n)} \P( \Omega_{n}^+\cap \{  (S_{n},A_{n-1})=(-k',l') \} ) \geq  \pr{\Omega_{n}^+} \kappa,
$$
where we used (\ref{eqn:nodr+}) in both cases. [To see the second relation one needs the following additional argument in conjunction with (\ref{eqn:nodr+}): Let $0<\eps<b-a$. Then for all $n$ large enough
\begin{align*}
&\P\left(  S_{n} \in [a n^{1/2},b n^{1/2}], A_{\mathbf{n-1}}\in [(a-\eps) n^{3/2},b n^{3/2}] ~ |~ \Omega_{n}^+ \right)\\
& \geq \P\left(  S_{n} \in [a n^{1/2},b n^{1/2}], A_{\mathbf{n}}\in [a n^{3/2},b n^{3/2}] ~ |~ \Omega_{n}^+ \right) \geq \kappa>0;
\end{align*}
and we continue to work with $a-\eps$ instead of $a$.]

Thus, assuming (\ref{eqn:nodr-}), using the local limit theorem (Proposition~\ref{prop:llt}) for the denominator, and Sinai's result \cite{sinai1991}, we get
\begin{align*}
& \P( \Omega_{4n}^+ ~|~  (S_{4n},A_{4n})=(0,0) )  = \frac{ \P( \Omega_{4n}^+ \cap \{  (S_{4n},A_{4n})=(0,0) \} )}{\pr{(S_{4n},A_{4n})=(0,0)}} \\ 
& \geq \frac{\kappa' n^{-2}  \cdot \kappa  c\, n^{-1/4} \cdot \kappa c\, n^{-1/4}}{c\, (4n)^{-2}} \approx n^{-1/2},
\end{align*}
which shows the assertion.

It remains to prove (\ref{eqn:nodr-}). We proceed with a proof by contradiction.

Assume that the $\liminf$ in (\ref{eqn:nodr-}) is zero. Then there exist an $\N$-valued sequence $(n_m)_{m\in\N}$  that tends to infinity and pairs $(k_m,l_m) \in \cL(n_m)$ and $(k'_m,l'_m)\in\mathcal R(n_m)$ for $m\in\N$ such that
$$
\lim_{m\to \infty} n_m^2\, \P^{(k_m,l_m)}( \Omega_{2n_m}^+ \cap \{  (S_{2n_m},A_{2n_m})=(k'_m,l'_m) \} ).
$$
 Without loss of generality we can assume that the limits
\begin{align*}
\mathfrak s:=\lim_{m\to\infty} ( (2n_m)^{-1/2} {k_m}, (2n_m)^{-3/2} {l_m})\in[a,b]^2, \\ \pin:=\lim_{m\to\infty} ((2n_m)^{-1/2}  k'_m, (2n_m)^{-3/2} l'_m) \in[a,b]^2
\end{align*}
exist, since this is the case for at least one subsequence of $(n_m)$. In order to apply he central limit theorem, we work with the continuous function 
$$
F:C[0,1]^2\to [0,\infty), (z^1,z^2)\mapsto 1\wedge (\inf_{t\in[0,1]} z_t^2)^+ \leq \ind_{\{z_t^2\geq 0, t\in[0,1]\}}.
$$
Let
$$
Z_t^{m} :=(Z_t^{m,1},Z_t^{m,2}):=( (2n_m)^{-1/2}\,S_{2n_m t},  (2n_m)^{-3/2}\,A_{2n_m t} )
$$
for $t\in \N_0/(2n_m)$ and extend the definition of $Z^m$ between the points in $\N_0/(2n_m)$  linearly. 
Then
\begin{align*}
 &\P^{(k_m,l_m)}( \Omega_{2n_m}^+ \cap \{  (S_{2n_m},A_{2n_m})=(k',l') \} )\notag  \\
&\geq \E^{(k_m,l_m)}[ F(Z) ~|~  (S_{2n_m},A_{2n_m})=(k'_m,l'_m) ]  \cdot \P((S_{2n_m},A_{2n_m})=(k'_m,l'_m) )
\end{align*}
Since $\P^{(k_m,l_m)} (S_{2n_m},A_{2n_m})=(k'_m,l'_m) )\geq  c (2n_m)^{-2}$ for a positive constant $c$, a contradiction is achieved once we show that 
$$
\limsup_{m\to\infty}   \E^{(k_m,l_m)}[ F(Z) ~|~  (S_{2n_m},A_{2n_m})=(k'_m,l'_m) ]>0.
$$
However, this follows directly from  the local limit theorem (Proposition~\ref{prop:llt}) and Remark~\ref{rem1}: the limsup is actually a limit and it is equal to
$$
\E^{\mathfrak s} [F(\Gamma)| \Gamma_1=\mathfrak p].
$$ 
Further it is positive, since $\mathfrak s_2$ and $\mathfrak p_2$ are positive.


\begin{thebibliography}{10}

\bibitem{aurzadafbm}
F.~Aurzada.
\newblock On the one-sided exit problem for fractional brownian motion.
\newblock {\em Electronic Communications in Probability}, 16:392--404, 2011.

\bibitem{aurzadabaumgarten}
F.~Aurzada and C.~Baumgarten.
\newblock Survival probabilities for weighted random walks.
\newblock {\em ALEA. Latin American Journal of Probability and Mathematical
  Statistics}, 8:235--258, 2011.

\bibitem{aurzadadereich1}
F.~Aurzada and S.~Dereich.
\newblock Universality of the asymptotics of the one-sided exit problem for
  integrated processes.
\newblock Preprint, arXiv:1008.0485, to appear in: {\it Ann.\ Inst.\ Henri
  Poincar\'e Probab.\ Stat.}, 2010.

\bibitem{assurvey}
F.~Aurzada and T.~Simon.
\newblock Persistence probabilities \& exponents.
\newblock Preprint, arXiv:1203.6554, 2012.

\bibitem{baumgarten2}
C.~Baumgarten.
\newblock Survival probabilities of some iterated processes.
\newblock Preprint, arXiv:1106.2999, 2011.

\bibitem{baumgarten3}
C.~Baumgarten.
\newblock Survival probabilities of autoregressive processes.
\newblock Preprint, 2012.

\bibitem{Bil99}
Patrick Billingsley.
\newblock {\em Convergence of probability measures}.
\newblock Wiley Series in Probability and Statistics: Probability and
  Statistics. John Wiley \& Sons Inc., New York, second edition, 1999.
\newblock A Wiley-Interscience Publication.

\bibitem{BS}
A.~Bulinski and A.~Shashkin.
\newblock {\em Limit theorems for associated random fields and related
  systems}.
\newblock Advanced Series on Statistical Science \& Applied Probability, 10.
  World Scientific Publishing Co. Pte. Ltd., Hackensack, NJ, 2007.

\bibitem{caravennadeuschel2009}
F.~Caravenna and J.-D. Deuschel.
\newblock Scaling limits of {$(1+1)$}-dimensional pinning models with
  {L}aplacian interaction.
\newblock {\em Ann. Probab.}, 37(3):903--945, 2009.

\bibitem{dembogao}
A.~Dembo and F.~Gao.
\newblock Persistence of iterated partial sums.
\newblock Preprint, arXiv:1101.5743, to appear in: {\it Ann.\ Inst.\ Henri
  Poincar\'e Probab.\ Stat.}, 2011.

\bibitem{esaryetal}
J.~D. Esary, F.~Proschan, and D.~W. Walkup.
\newblock Association of random variables, with applications.
\newblock {\em Ann. Math. Statist.}, 38:1466--1474, 1967.

\bibitem{feller2}
W.~Feller.
\newblock {\em An introduction to probability theory and its applications.
  {V}ol. {II}.}
\newblock Second edition. John Wiley \& Sons Inc., New York, 1971.

\bibitem{majumdar}
S.N. Majumdar.
\newblock Persistence in nonequilibrium systems.
\newblock {\em Current Science}, 77(3):370--375, 1999.

\bibitem{sinai1991}
Ya.~G. Sina{\u\i}.
\newblock Distribution of some functionals of the integral of a random walk.
\newblock {\em Teoret. Mat. Fiz.}, 90(3):323--353, 1992.

\bibitem{vlad1}
V.~Vysotsky.
\newblock On the probability that integrated random walks stay positive.
\newblock {\em Stochastic Process. Appl.}, 120(7):1178--1193, 2010.

\bibitem{vlad2}
V.~Vysotsky.
\newblock Positivity of integrated random walks.
\newblock Preprint, arXiv:1107.4943, 2011.

\end{thebibliography}

\end{document}